\documentclass[12pt]{amsart}
\usepackage[all]{xy}
\usepackage{relsize}
\usepackage{amssymb}
\usepackage{amsthm}
\usepackage{hyperref}
\usepackage{anyfontsize}
\hypersetup{colorlinks=true,linkcolor=blue,citecolor=magenta}
\usepackage{amsmath}
\usepackage{amscd,enumitem}
\usepackage{verbatim}
\usepackage{eurosym}
\usepackage{float}
\usepackage{color}
\usepackage{dcolumn}
\usepackage[mathscr]{eucal}
\usepackage[all]{xy}
\usepackage{hyperref}
\usepackage{bbm}
\usepackage[textheight=8.8in, textwidth=6.8in]{geometry}
\usepackage{mathtools}

\DeclarePairedDelimiter{\floor}{\lfloor}{\rfloor}
\newtheorem*{thm*}{Theorem}
\newtheorem*{conj*}{Conjecture}

\newtheorem{theorem}{Theorem}[section]
\newtheorem{lemma}{Lemma}[section]

\newtheorem{proposition}[theorem]{Proposition}

\newtheorem*{rmk}{Remark}

\newtheorem*{example}{Example}

\newcommand{\Z}{\mathbb{Z}}
\newcommand{\Q}{\mathbb{Q}}
\newcommand{\R}{\mathbb{R}}

\newcommand{\N}{\mathbb{N}}

\newcommand{\Vol}{\operatorname{Vol}}
\numberwithin{equation}{section}

\begin{document}
\title{Modular Forms and Ellipsoidal $T$-Designs}
\author{Badri Vishal Pandey}
\address{Department of Mathematics, University of Virginia, Charlottesville, VA 22904}
\email{bp3aq@virginia.edu}

\begin{abstract} In recent work, Miezaki introduced the notion of a {\it spherical $T$-design} in $\R^2$, where $T$ is a potentially infinite set. As an example, he offered the $\Z^2$-lattice points with fixed integer norm (a.k.a. shells). These shells are $maximal$ spherical $T$-designs, where $T=\Z^+\setminus 4\Z^+$. We generalize the notion of a spherical $T$-design to special ellipses, and extend Miezaki's work to the norm form shells for rings of integers of imaginary quadratic fields with class number 1.
\end{abstract}

\maketitle
\section{Introduction and statement of results}\label{Intro}
$Spherical$ $t$-$designs$ were introduced in 1977 by Delsarte, Goethals and Seidel \cite{Delsarte}, and they have played an important role in algebra, combinatorics, number theory and quantum mechanics (for background see \cite{Bannai}, \cite{BOT}, \cite{Chen}, \cite{quantum}, \cite{Go}, \cite{Miezaki}). A spherical $t$-design is a nonempty finite set of points on the unit sphere with the property that the average value of any real polynomial of degree $\leq t$ over this set equals the average value over the sphere. Namely, if $S^{n-1}$ denotes the unit sphere in $\mathbb{R}^n$ centered at the origin, then a finite nonempty subset $X\subset S^{n-1}$ is a spherical $t$-design if
\begin{equation}\label{E1}
    \dfrac{1}{|X|}\sum_{x\in X}P(x)= \dfrac{1}{\Vol(S^{n-1})}\int_{S^{n-1}} P(x)d\sigma(x)
\end{equation}
for all polynomials $P(x)$ of degree $\leq t$. The right-hand side of ($\ref{E1}$) is the usual surface integral over $S^{n-1}$. In general, a finite nonempty subset $X$ of $S_{n-1}(r)$, the sphere of radius $r$ centered at the origin, is a spherical $t$-design if $\frac{1}{r}X$ satisfies (\ref{E1}). Since a spherical $t$-design is also a spherical $t^{\prime}$-design for all $ t^{\prime}\leq t$, we say that $X$ has $strength$ $t$ if it is the maximum of all such numbers.

Delsarte, Goethals and Seidel developed a very simple criterion for determining spherical $t$-designs. This criterion involves {\it homogeneous harmonic} polynomials of bounded degree. A polynomial in $n$ variables is $harmonic$ if it is annihilated by the Laplacian operator $\Delta:=\sum^{n}_{i=1}\partial^2/\partial x^2_i$, and they showed \cite{Delsarte} that $X\subset S^{n-1}$ is a spherical $t$-design if 
\begin{equation}\label{P}
    \sum_{x\in X}P(x) = 0
\end{equation}
for all homogeneous harmonic polynomials $P(x)$ of nonzero degree $\leq t$. This criterion is a consequence of two results from harmonic analysis. The first result is the mean value property for harmonic functions \cite[p.~5]{HFT}, which implies that the integral of a harmonic polynomial over a sphere centered at the origin vanishes, combined with the fact that homogeneous polynomials of fixed degree are spanned by certain harmonic polynomials \cite[Th. ~5.7]{HFT}.

In view of this framework, it is natural to ask whether there are generalizations of spherical $t$-designs to other curves, surfaces and varieties. Here we consider certain $ellipsoids\footnote{We do not use the term $ellipse$ to avoid possible confusion that might arise with the term $elliptical$.}$ in dimension two. To be precise, for square-free $D\geq 1$ we define the norm $r$ ellipses 
\begin{equation}
    C_D(r) :=
\begin{cases}
\{(x,y)\in \R^2 : x^2+Dy^2=r\} & \text{ if } D\equiv 1,2\pmod{4},\\
\{(x,y)\in \R^2 : x^2+xy+\frac{1+D}{4}y^2=r\} & \text{ if } D\equiv 3\pmod{4}.
\end{cases}
\end{equation}
\begin{rmk}
These ellipses arise from certain imaginary quadratic orders.
\end{rmk}
For $D\equiv 1,2\pmod 4$, we say that a finite nonempty subset $X\subset C_D(r)$ is an $ellipsoidal$ $t$-$design$ if
\begin{equation}\label{HE1}
    \dfrac{1}{|X|}\sum_{(x,y)\in X}P(x,y)=
\dfrac{1}{2\pi\sqrt{D}}\mathlarger\int_{C_D(r)}\dfrac{P(x,y)}{\sqrt{x^2/D^2+y^2}}d\sigma(x,y) 
\end{equation} for all polynomials $P(x,y)$ of degree $\leq t$ over $\R$.
For $ D\equiv 3\pmod 4$, instead we require
\begin{equation}\label{HE2}
\dfrac{1}{|X|}\sum_{(x,y)\in X}P(x,y)=
\dfrac{\sqrt{D}}{\pi}\mathlarger\int_{C_D(r)}\dfrac{P(x,y)}{\sqrt{20x^2+(D^2+2D+5)y^2+(20+4D)xy}}d\sigma(x,y).
\end{equation}
Here the right-hand sides are line integrals. As in the case of spherical $t$-designs, every ellipsoidal $t$-design is also an ellipsoidal $t^{\prime}$-design for all $ t^{\prime}\leq t$, and the maximum of all such $t$'s is called the $strength$ of $X$.These definitions coincide with the notion of a spherical $t$-design when $D=1$. 

In analogy to Delsarte, Goethals and Seidel, we have a natural criterion for confirming ellipsoidal $t$-designs. To this end, we consider the 2-dimensional real vector space \begin{equation}\label{D-Harmonic}
   H^{\R}_{D,j}[x,y] := 
\begin{cases}
\langle \text{Re}(x+\sqrt{-D}y)^j,\text{Im}(x+\sqrt{-D}y)^j\rangle & \text{ if } D\equiv 1,2 \pmod{4},\\
\langle \text{Re}(x+\frac{1+\sqrt{-D}}{2}y)^j,\text{Im}(x+\frac{1+\sqrt{-D}}{2}y)^j\rangle & \text{ if } D\equiv 3 \pmod{4}.
\end{cases} 
\end{equation}
In terms of these vector spaces of polynomials, we have the following ellipsoidal $t$-design criterion.
\begin{theorem}\label{CD}
A finite nonempty set $X\subset C_D(r)$ is an ellipsoidal $t$-design if \begin{equation*}
    \sum_{x\in X}P(x,y)=0
\end{equation*}
for all $P(x,y)\in H^{\R}_{D,j}[x,y]$ for all $0<j\leq t$.
\end{theorem}
\begin{rmk}
1)Observe that if $X\subset S^{1}$ is a spherical $t$-design, then $Y=\{(x,y/\sqrt{D})|(x,y)\in X\}\subset C_D$ (resp. $Y=\{(x+y/\sqrt{D},2y/\sqrt{D}|(x,y)\in X\}\subset C_D$) is an ellipsoidal $t$-design for $D\equiv 1,2 \pmod{4}$ (resp. $D\equiv 3 \pmod{4}$). Therefore, the existence of a spherical $t$-design implies the existence of a corresponding ellipsoidal $t$-design. In fact, there is a one-to-one correspondence between spherical $t$-designs and ellipsoidal $t$-designs. However, the proof of Theorem \ref{CD} is not a direct consequence because care is required for justifying the role of the vector spaces $H^{\R}_{D,j}[x,y]$.\\
2)Since there is one-to-one correspondence between spherical and ellipsoidal $t$-designs, we get a lower bound \cite[pg 2]{Delsarte} on the size of ellipsoidal $t$-design $X$,
$$|X|\geq t+1. $$
\end{rmk}

Recently, Miezaki in \cite{Miezaki} introduced a generalization of the notion of spherical $t$-designs. Instead of restricting to polynomials of degree $\leq t$, he considered harmonic polynomials of degree $j\in T\subset \N$, where $T$ is a potentially infinite set. The main theorem from \cite{Miezaki} gives infinitely many spherical $T$-designs for $T :=\Z^{+}\setminus 4\Z^+$ in dimension two. Namely, he considered norm $r$ shells, integer points on $x^2+y^2=r$ for fixed $r\in \Z^+$. He showed that these $r$-shells are spherical $T$-designs. Moreover, these sets have strength $T$, meaning that (\ref{P}) fails if any multiple of 4 is added to $T$. His proof makes use of theta functions arising from complex multiplication by $\mathbb{Z}[i]$. 

We generalize Miezaki's work to ellipsoidal $T$-designs. We call $X\subset C_D$ an $ellipsoidal$ $T$-$design$ if the condition in Theorem \ref{CD} is satisfied for all polynomials in $H^{\R}_{D,j}[x,y]$ with $j\in T$. We say $X$ has strength $T$ if it is maximal among such sets. For each square-free positive integer $D$, let $\mathcal{O}_D$ be the ring of integers of $\Q(\sqrt{-D})$. In particular, this means that
\begin{equation}\label{OD}
    \mathcal{O}_D=\begin{cases}
\Z[\sqrt{-D}] \ \ \ \ \ & \text{ if } D\equiv 1,2 \pmod{4}, \\
\Z[\frac{1+\sqrt{-D}}{2}] \ \ \ \ \  & \text{ if } D\equiv 3 \pmod{4}.
\end{cases}
\end{equation}
We consider $D\in\{1,2,3,7,11,19,43,67,163\}$, the square-free positive integers for which $\mathcal{O}_D$ has class number 1. To make this precise, we define the {\it norm $r$ shells} in
$C_D(r)$ by
\begin{equation}\label{NormShells}
\Lambda_D^r:= \mathcal{O}_D \cap C_D(r).
\end{equation}
Generalizing Miezaki’s work for $D=1$, we obtain the following theorem.
\begin{theorem}\label{Main}
If $D\in \{1, 2, 3, 7, 11, 19, 43, 67, 163\},$ then every non-empty shell
$\Lambda_D^r$ is an ellipsoidal $T_D$ design with strength $T_D$, where
$$
T_D:=\begin{cases} \Z^+\setminus 4\Z^+ \ \ \ \ \ &{\text {\rm if}}\ D=1,\\
               \Z^+\setminus 6\Z^{+} \ \ \ \ &{\text {\rm if}}\ D=3,\\
              \Z^+\setminus 2\Z^+ \ \ \ \ &{\text {\rm otherwise.}}
\end{cases}
$$
\end{theorem}
\begin{rmk}
The method used here seems to be well-poised only for the dimension 2 cases. 
It would be interesting to obtain higher dimensional analogues.
\end{rmk}
\begin{example}
We consider $D=3$, and $r=691$. Then we have
\begin{align*}
    \Lambda^{691}_3=&\{(11,19),(-11,-19),(19,11),(-19,-11),(11,-30),(-11,30),(30,-19),(-30,19),\\
    &(30,-11),(-30,11),(19,-30),(-19,30)\}.
\end{align*}

We consider the polynomial $P(x,y)=2x^2+3462xy+1729y^2\in H^{\R}_{3,2}[x,y]$, and we find that $\sum_{(x,y)\in \Lambda^{691}_3}P(x,y)=0$ which shows that $\Lambda^{691}_3$ is an elliptical $2$-design and $2\in T_3$.
On the other hand, Theorem \ref{Main} implies that $\Lambda^{691}_3$ is not an ellipsoidal $6$-design.To see this we choose $Q(x,y)=2x^2+6x^5y-15x^4y^2-40x^3y^3-15x^2y^4+6xy^5+2y^6\in H^{\R}_{3,6}(x,y),$ and we find that $\sum_{(x,y)\in \Lambda^{691}_3}Q(x,y)=-4818834696\not=0$.
\end{example}
In Section $2$ we prove Theorem \ref{CD}, criterion for confirming that a set is an ellipsoidal $t$-design, and in Section $3$ we recall the theory of theta functions arising from complex multiplication, and we prove Theorem \ref{Main}.
\section*{Acknowledgement}
I would like to thank Prof Ken Ono for suggesting me this problem and guiding through. I also thank Will Craig and Wei-Lun Tsai for reviewing my paper and giving useful comments. I thank Matthew McCarthy for helping me with Sage Math. Lastly, I would like to thank the reviewers for their useful comments. 
\section{Criterion for ellipsoidal $t$-Design}
In this section we prove Theorem \ref{CD}, criterion for confirming ellipsoidal $t$-designs. Throughout this section we assume that $D\geq 1$ is square-free and $j\geq 1$. 

To prove that Theorem \ref{CD} is indeed a criterion for confirming ellipsoidal $t$-designs, we first need to show that the spaces $H^{\R}_{D,k}[x,y]$, for $0<k\leq j$, generate all the polynomials of degree $\leq j$ when restricted to $C_D(r)$. It suffices to show this for $P^{\R}_j[x,y]$, the set of homogeneous polynomials of degree $j$.
\begin{lemma}\label{H-Span}
If $D\geq 1$ is square-free and $j\geq 1$, then the following are true:\\
1) If $D\equiv 3\mod{4}$, then we have
 \begin{equation*}
    P^{\R}_j[x,y] = \bigoplus^{\floor{j/2}}_{k=0} (x^2+Dy^2)^{k}H^{\R}_{D,j-2k}[x,y].
    \end{equation*}
2) If $D\equiv 3\mod{4}$, then we have
\begin{equation*}
    P^{\R}_j[x,y] = \bigoplus^{\floor{j/2}}_{k=0} \Big(x^2+xy+\frac{1+D}{4}y^2\Big)^kH^{\R}_{D,j-2k}[x,y].
\end{equation*}
\end{lemma}
\begin{proof}
The lemma is well known for homogeneous harmonic polynomials (for example, see \cite[Thm~5.7]{HFT}). Namely, if $H^{\R}_{k}[x,y]$ is the set of homogeneous harmonic polynomials of degree $k$ then 
\begin{equation*}
    P^{\R}_j(x,y) = \bigoplus^{\floor{j/2}}_{k=0} (x^2+y^2)^{k}H^{\R}_{j-2k}[x,y].
    \end{equation*}

We extend it to general $D$.
It is well known that $H^{\R}_j[x,y]=\langle\text{Re}(x+iy)^j, \text{Im}(x+iy)^j\rangle$, and so if we do the change of variable for $D\equiv 1,2\mod{4}$ (resp. $D\equiv 3\mod{4}$),
$x'=x$,$y'=\sqrt{D}y$ (resp.  $x'=x+y/2$,$y'=2y/\sqrt{D}$), then $H^{\R}_{j-2}(x',y')=\langle \text{Re}(x'+ iy')^j,\text{Im}(x'+iy')^j\rangle$ gives
\begin{equation*}
    P^{\R}_j[x',y'] = \bigoplus^{\floor{j/2}}_{k=0} (x'^2+y'^2)^{k}H^{\R}_{j-2k}[x',y'].
    \end{equation*}
Therefore, if $D\equiv 1,2\mod{4}$, then we have
\begin{equation*}
    P^{\R}_j(x,y) = \bigoplus^{\floor{j/2}}_{k=0} (x^2+Dy^2)^{k}H^{\R}_{D,j-2k}[x,y].
    \end{equation*}
If $D\equiv 3\mod{4}$, then we have
\begin{equation*}
    P^{\R}_j(x,y) = \bigoplus^{\floor{j/2}}_{k=0} \Big(x^2+xy+\frac{1+D}{4}y^2\Big)^kH^{\R}_{D,j-2k}[x,y].
\end{equation*}
\end{proof}

We now prove Theorem \ref{CD}.
\begin{proof}[Proof of Theorem \ref{CD}]\label{IN} Lemma \ref{H-Span} shows that the set of polynomials when restricted to $C_D$ are generated by the spaces $H^{\R}_{D,j}[x,y]$ since $x^2+Dy^2=r$ (resp., $x^2+xy+\frac{1+D}{4}y^2=r$) on $C_D(r)$. 
Therefore, it suffices to show that if $P(x,y)\in H^{\R}_{D,j}[x,y]$, then the following are true:\\
    1) If $D\equiv 1,2\mod{4}$, then we have
    $$\int_{C_D(r)} \dfrac{P(x,y)}{\sqrt{x^2/D^2+y^2}}d\sigma(x,y)=0.$$
    2) If $D\equiv 3\mod{4}$, then we have
    $$\int_{C_D(r)}\dfrac{P(x,y)}{\sqrt{20x^2+(D^2+2D+5)y^2+(20+4D)xy}}d\sigma(x,y)=0.$$
As $H^{\R}_{D,j}[x,y]$ is a vector space, it is enough to show these claims for basis vectors. Since $X\subset C_D(r)$ is an ellipsoidal $t$-design if and only if $\frac{1}{r}\subset C_D(1)$ is an ellipsoidal $t$-design, it's enough to consider $r=1$. 
For $D\equiv 1,2\pmod{4}$,
$H^{\R}_{D,j}[x,y]=\langle\text{Re}(x+\sqrt{-D}y)^j,\text{Im}(x+\sqrt{-D}y)^j\rangle$. By the parametrization of $C_D(1):x^2+Dy^2=1$ as $\gamma:=\{(\cos{\theta},\sin{\theta}/\sqrt{D})|0\leq \theta\leq 2\pi\}$, we have
\begin{align*}
    \int_{C_D(1)} \dfrac{\text{Re}(x+\sqrt{-D}y)^j}{\sqrt{x^2/D^2+y^2}}d\sigma(x,y) &=  \int^{2\pi}_{0}\dfrac{\text{Re}(\cos{\theta}+\sqrt{-D}(\sin{\theta}/\sqrt{D}))^j}{\sqrt{\cos{\theta}^2/D^2+\sin{\theta}^2/D}}\sqrt{\sin{\theta}^2+\cos{\theta}^2/D}d\theta \\
    &=\sqrt{D}\int^{2\pi}_{0}\text{Re}(\cos{\theta}+i\sin\theta)^jd\theta 
    =\sqrt{D}\int_{S^1}\text{Re}(x+i y)^jdz =0.
\end{align*}
Since $\text{Re}(x+i y)^j$ is harmonic, the last integral over $S^1$ is $0$.

A similar argument shows that 
$$\int_{C_D(1)} \dfrac{\text{Im}(x+\sqrt{-D}y)^j}{\sqrt{x^2/D^2+y^2}}d\sigma(x,y)=0.$$
If $D\equiv 3\pmod{4}$,
$H^{\R}_{D,j}[x,y]=\langle\text{Re}(x+\frac{1+\sqrt{-D}}{2}y)^j,\text{Im}(x+\frac{1+\sqrt{-D}}{2}y)^j\rangle$. By the parametrization of $C_D(1):x^2+xy+\frac{1+D}{4}y^2=1$ as $\gamma:=\{(\cos{\theta}-\sin{\theta}/\sqrt{D},2\sin{\theta}/\sqrt{D}):0\leq \theta\leq 2\pi\}$, we have
\begin{align*}
      & \int_{C_D(1)} \dfrac{\text{Re}(x+(1+\sqrt{-D})y/2)^j}{\sqrt{20x^2+(D^2+2D+5)y^2+(20+4D)xy}}d\sigma(x,y)\\  = &\int^{2\pi}_{0}\frac{\text{Re}(\cos{\theta}-\sin{\theta}/\sqrt{D}+(1+\sqrt{-D}\sin{\theta}/\sqrt{D})^j}{\sqrt{4D\sin{\theta}^2+20\cos{\theta}^2+8\sqrt{D}\sin{\theta}\cos{\theta}}}\sqrt{\sin{\theta}^2+5\cos{\theta}^2/D+2\sin{\theta}\cos{\theta}/\sqrt{D}}d\theta \\
      = &\dfrac{1}{2\sqrt{D}}\int^{2\pi}_{0}\text{Re}(\cos{\theta}+i\sin\theta)^jd\theta 
      = \dfrac{1}{2\sqrt{D}}\int_{S^1}\text{Re}(x+i y)^jdz=0.
\end{align*}

A similar argument shows that $$\int_{C_D(1)}\dfrac{P(x)}{\sqrt{20x^2+(D^2+2D+5)y^2+(20+4D)xy}}d\sigma(x,y)=0.$$
\end{proof}

\section{ellipsoidal T-Designs}
Here we prove Theorem \ref{Main}, the construction of ellipsoidal $T$-designs arising from the ring of integers of imaginary quadratic fields with class number 1. We use the theory of theta functions with complex multiplication. Throughout, we shall assume that $D\in\{1,2,3,7,11,19,43,67,163\}.$
\subsection{Theta functions}
Given an $n$-dimensional lattice $\Lambda$ and a polynomial $P(x)$ of degree $j$ in $n$ variables, the theta function of $P(x)$ over the lattice $\Lambda$ is defined by the Fourier series (note $q:=e^{2\pi iz}$) 
\begin{equation}
    \Theta(\Lambda,P;z):=\sum_{x\in\Lambda} P(x)q^{N(x)} = \Theta(\Lambda,P;z)=\sum^\infty_{n=0} a(\Lambda,P,n)q^n, 
\end{equation}
where $N(x)$ is the standard norm in $\R^n$.
The theta functions for $\Lambda_D= \mathcal{O}_{D}$ play an important role in the study of ellipsoidal $T$-designs. Namely, if 
$\Theta(\Lambda_D,P;z)=\sum^\infty_{r=0} a(\Lambda_D,P,r)q^r,$
then
\begin{equation}\label{Coeff}
    a(\Lambda_D,P,r)= \mathlarger\sum_{(x,y)\in \Lambda^r_D}P(x,y).
\end{equation}
The theta function $\Theta(\Lambda_D,P;z)\in \mathcal{M}_k(\Gamma_0(4D),\chi)$, the space of holomorphic modular forms with weight $k=j+1$ and nebentypus $\chi(A)=(\frac{-D}{d}),$ where $A=\Big(\begin{array}{cc}
    a & b \\
    c & d
\end{array}\Big)$ \cite[Thm~10.8]{AF}. Moreover, $\Theta(\Lambda_D,P;z)$ is a cusp form when $j>0$.

To ease the study of these theta function, it is convenient to introduce the following the polynomials for each $j\geq 1$: 
\begin{equation}
    R_{D,j}(x,y):=\begin{cases}
    \text{Re}(x+\sqrt{-D}y)^j \ \ \ \ \ & \text{ if } D\equiv 1,2\pmod{4}, \\
    \text{Re}(x+\frac{1+\sqrt{-D}}{2}y)^j \ \ \ \ \ & \text{ if } D\equiv 3\pmod{4}, \\
    \end{cases}
\end{equation}
and \begin{equation}
    I_{D,j}(x,y):=\begin{cases}
    \text{Im}(x+\sqrt{-D}y)^j \ \ \ \ \ & \text{ if } D\equiv 1,2\pmod{4}, \\
    \text{Im}(x+\frac{1+\sqrt{-D}}{2}y)^j \ \ \ \ \ & \text{ if } D\equiv 3\pmod{4}. \\
    \end{cases}
\end{equation}
By definition, we have that $H^{\R}_{D,j}[x,y]=\langle R_{D,j}(x,y),I_{D,j}(x,y)\rangle$. In particular, $\Theta(\Lambda_D,R_{D,j};z)$ and $\Theta(\Lambda_D,I_{D,j};z)$ are cusp forms. Theorem \ref{CD} together with the discussion above gives the following lemma which transforms the problem of determining ellipsoidal $T$-designs into the vanishing of certain coefficients of special theta functions.
\begin{lemma}\label{Equiv}
The norm $r$ shell $\Lambda^r_D=\Lambda_D\cap C_D(r)$ is an ellipsoidal $T$-design if and only if $a(\Lambda_D,R_{D,j},r)=0$ and $a(\Lambda_D,I_{D,j},r)=0$ for all $j\in T$.
\end{lemma}
We require some standard facts from the theory of newforms. Since $\mathcal{O}_D$ has class number 1, each {\it Hecke character} mod $\mathcal{O}_D$ is defined by its values on principal ideals. Let $(\alpha)\subset \mathcal{O}_D$ be a principal ideal. Let $u_D$ be the number of units in $\mathcal{O}_D$, namely
\begin{equation}
    u_D:=\begin{cases}
    4 \ \ \ \ \ & \text{ if } D=1, \\
    6 \ \ \ \ \ & \text{ if } D=3, \\
    2 \ \ \ \ \ & \text{ otherwise. }
    \end{cases}
\end{equation}
For each positive $j_D\equiv 0\pmod{u_D}$, define Hecke characters mod $\mathcal{O}_D$ by:
$$\zeta_{j_D}((\alpha))=\Big(\dfrac{\alpha}{|\alpha|}\Big)^{j_D}$$
Then by \cite[Thm~4.8.2]{Miyake}, we have the following well known lemma about the modular form 
\begin{equation*}
    f_{j_D}(\zeta_{j_D};z):=\begin{cases}
    \Theta(\Lambda_D,(x+\sqrt{-D}y)^j;z) & \text{ if } D\equiv 1,2\pmod{4}, \\
    \Theta\Big(\Lambda_D,\Big(x+\frac{1+\sqrt{-D}}{2}y\Big)^j;z\Big) & \text{ if } D\equiv 3\pmod{4}
    \end{cases}
\end{equation*}
\begin{lemma}\label{newform}
Assuming the notations above, we have
\begin{equation*}
    f_{j_D}(\zeta_{j_D};z)=\sum_{(\alpha)\subset \mathcal{O}_D} \zeta_{j_D}((\alpha))N(\alpha)^{j/2} q^{N(\alpha)} \in  \mathcal{S}_{k_D}(\Gamma_0(N),\chi),
\end{equation*}
the space of cusp forms of weight $k_D=j_D+1$ with nebentypus $\chi\pmod{N}$. Here $N:=|\Delta_{\mathcal{O}_D}|$, the absolute value of the discriminant of $\mathcal{O}_D$. Moreover, $f_{j_D}(\zeta_{j_D};z)$ is a {\it newform.}
\end{lemma}
\subsection{Other Propositions and Lemmas}
Recall that $\Lambda^r_D=C_D(r)\cap \mathcal{O}_D$. Using well known facts about the positive definite binary quadratic forms corresponding to class number 1 norm forms, we have the following lemma.
\begin{lemma}\label{prime-decom}
Suppose $r$ is a positive integer. Then $\Lambda^r_D$ is nonempty if and only if $ord_p(r)$ is even for every prime $p\nmid r$ for which $\Lambda^p_D$ is nonempty.
\end{lemma}
Rewriting (\ref{Coeff}), we have
\begin{equation}
    a(\Lambda_D,P,r)=
\mathlarger\sum_{(x,y)\in\Lambda^r_D}P(x,y).
\end{equation}
Lemma \ref{Equiv} implies that $\Lambda^r_D$ is an ellipsoidal $T$-design if and only if $a(\Lambda_D,R_{D,j},r)$ and $a(\Lambda_D,I_{D,j},r)$ vanish for all $j\in T$. Since $\Lambda^r_D$ is antipodal ($i.e.$ $-\Lambda^r_D=\Lambda^r_D$ for all $r$), $a(\Lambda_D,R_{D,j},r)$ and $a(\Lambda_D,I_{D,j},r)$ are 0 for all $j\in \Z^+\setminus2\Z^+$. Therefore, we have that following proposition.
\begin{proposition}\label{odd}
Suppose $r\in \Z^+$ such that $\Lambda^r_D$ is nonempty. Then $\Lambda^r_D$ is an ellipsoidal $\Z^+\setminus 2\Z^+$-design.
\end{proposition}
Our objective is to find maximal set $T_D$ for which $\Lambda^r_D$ is ellipsoidal $T$-design. By proposition above we have that $\Z^+\setminus 2\Z^+\subset T_D$. So we only look for all even $j$ which can be in $T_D$. 
\begin{proposition}\label{I}
Suppose $j\equiv 0\pmod{2}$, and $r\in \Z^+$. Then the following are true:\\
1) We have that $a(\Lambda_D,I_{D,j},r)=0.$\\ 
2) We have that $a(\Lambda_D,R_{D,j},r)=
\begin{cases}
\mathlarger\sum_{(x_0,y_0)\in \Lambda^r_D}(x+\sqrt{-D}y)^j \ \ \ \ & \text{\rm if } D\equiv 1,2\pmod{4}, \\
\mathlarger\sum_{(x_0,y_0)\in \Lambda^r_D}\Big(x+\frac{1+\sqrt{-D}}{2}y\Big)^j \ \ \ \ & \text{\rm if } D\equiv 3\pmod{4}
\end{cases}$
\end{proposition}
\begin{proof}
Part{\it (2)} is an obvious consequence of part{\it (1)}. So it is enough to prove part{\it (1)}. The idea is to show that points in $\Lambda^r_D$ occur in pairs on which value of $I_{D,j}$ cancel. If $D\equiv 1,2\pmod{4}$, then $I_{D,j}={\rm Im}(x+\sqrt{-D}y)^j$. In this case  $(a,b), (a,-b)\in \Lambda^r_D$ such that $I_{D,j}(a,b)+I_{D,j}(a,-b)=0$. This is true because each term of $I_{D,j}(x,y)$ has odd power in both the variables $x,y$.
If $D\equiv 3\pmod{4}$, then $I_{D,j}={\rm Im}((x+\frac{1}{2}y)+\frac{\sqrt{-D}}{2}y)^j$. In this case $(a,b),(a+b,-b)\in \Lambda^j_D$ such that $I_{D,j}(a,b)+I_{D,j}(a+b,-b)=0$. This is because each term of $I_{D,j}(x,y)$ has odd power in $x+y/2,y.$
\end{proof}
We notice that if $(x_0,y_0)\in \mathcal{O}_D,$ then we have
\begin{equation}\label{unit}
    \sum_{\alpha_D \in \mathcal{O}_D:|\alpha_D|=1}R_{D,j}(\alpha_D(x_0,y_0))=R_{D,j}(x_0,y_0)\sum_{\alpha_D \in \mathcal{O}_D:|\alpha_D|=1}\alpha^j_D.
\end{equation}
\begin{proposition}\label{R}
If $r\geq 1$, $1\leq j\not\equiv 0\pmod{u_D}$, and $\Lambda^r_D$ nonempty, then $a(\Lambda^r_D,R_{D,j},r)=0$
\end{proposition}
\begin{proof} The idea is that if $(x_0,y_0)\in \Lambda^r_D$ then $\alpha_D(x_0,y_0) \in \Lambda^r_D$ where $\alpha_D$ is a unit in $\mathcal{O}_D$.
Therefore enough to show that the sum in RHS of (\ref{unit}) is 0.
For $D=1$, number of units in $\mathcal{O}_D,$ $u_D=4$ which are $\{1,-1,i,-i\}$. We have $1^j+(-1)^j+i^j+(-i)^j=0.$
For $D=3$, number of units in $\mathcal{O}_D,$ $u_D=6$ which are $\{ \pm 1 , \frac{\pm 1\pm\sqrt{-3}}{2}\}$. A brute force calculation shows the result.
For other $D$, the number of units in $\mathcal{O}_D,$ $u_D=2$ which are $\{1,-1\}$. For all $j$ odd, $(1)^j+(-1)^j=0$
\end{proof}
From here on we will only consider the theta function $\Theta\Big(\Lambda_D,\frac{1}{u_D}R_{D,j};z\Big)$ so let's give its coefficients a shorthand.
\begin{equation}\label{Theta-short}
    \Theta\Big(\Lambda_D,\frac{1}{u_D}R_{D,j};z\Big)=\sum^\infty_{r=0} a(D,j,r)q^r
\end{equation}
Proposition \ref{I} together with Lemma \ref{newform} give us that if $j\equiv0\pmod{u_D}$, then the theta function $\Theta\Big(\Lambda_D,\frac{1}{u_D}R_{D,j};z\Big)\in \mathcal{S}_{j+1}(\Gamma_0(N),\chi)$ is a Hecke eigenform. So we have the following lemma.
\begin{lemma}\label{Hecke prop}
Suppose $j\in u_D\Z^+$. Then the following is true:\\
1) If $\gcd{(r_1,r_2)}=1$ then
$$
    a(D,j,r_1r_2)=a(D,j,r_1)a(D,j,r_2)
$$
2) For $p$ prime and $\alpha>0$, we have
$$
    a(D,j,p^{\alpha})= a(D,j,p)a(D,j,p^{\alpha-1})-\chi(p)p^ja(D,j,p^{\alpha-2})
$$
3) For $p$ prime and $\alpha>0$, we have
$$ a(D,j,p^{\alpha})= a(D,j,p)^{\alpha}\pmod{p}
$$
\end{lemma}
Suppose $p$ be a prime such that $\Lambda^p_D$ be nonempty. Let  $(x_p,y_p)\in \Lambda^p_D$ and $j\equiv 0\pmod{u_D}$. When $p=D$ then it ramifies in $\mathcal{O}_D$ and there are exactly $u_D$ points in $\Lambda^p_D$. From (\ref{unit}) we have $a(D,j,p)=R_{D,j}(x_p,y_p).$ If $p\not=D$ then it's unramified and we get exactly $2u_D$ solutions. In this case $a(D,j,p)=2R_{D,j}(x_p,y_p).$ 

\begin{lemma}\label{non-zero}
Suppose $j\in u_D\Z^+$ and $p$ be an odd prime such that $\Lambda^p_D$ is nonempty. Let $(x_p,y_p)\in \Lambda^p_D$ then $R_{D,j}(x_p,y_p)\not\equiv 0\pmod{p}$. In particular, $a(D,j,p)$ is non-zero.
\end{lemma}
\begin{proof}
We will consider two cases, $D\equiv 1,2\pmod{4}$ and $D\equiv 3\pmod{4}$. Proof is essentially same in both the cases.\\
If $D\equiv 1,2 \pmod{4}$ then $p=x_p^2+Dy_p^2$, in particular $x_p\not\equiv 0\pmod{p}$. we consider the binomial expansion
\begin{align*}
    R_{D,j}(x_p,y_p) &= \text{Re}(x_p+\sqrt{-D}y_p)^j \\
                     &= \frac{1}{2}\sum^{j/2}_{n=0}\binom{j}{2n}x^{j-2n}_p(-1)^n(Dy^2_p)^n 
                     = \frac{1}{2}\sum^{j/2}_{n=0}\binom{j}{2n}x^{j-2n}_p(-1)^n(p-x^2_p)^n \\
                     &\equiv \frac{1}{2}x^j_p\sum^{j/2}_{n=0}\binom{j}{2n} \equiv 2^{j-2}x^j_p \not\equiv 0 \pmod{p} 
\end{align*}
If $D\equiv 1,2 \pmod{4}$ then $p=(x_p+y_p/2)^2+Dy_p^2/4$, in particular $x_p+y_p/2\not\equiv 0\pmod{p}$. we consider the binomial expansion
\begin{align*}
    R_{D,j}(x_p,y_p) &= \text{Re}\Big(x_p+y_p/2+\sqrt{-D}y_p/2\Big)^j
                     = \frac{1}{2} \sum^{j/2}_{n=0}\binom{j}{2n}\Big(x_p+\frac{y_p}{2}\Big)^{j-2n}(-1)^n\Big(\frac{Dy^2_p}{4}\Big)^n \\
                     &= \frac{1}{2} \sum^{j/2}_{n=0}\binom{j}{2n}\Big(x_p+\frac{y_p}{2}\Big)^{j-2n}(-1)^n\Big(p-\Big(x_p+\frac{y_p}{2}\Big)^2\Big)^n  \\
                     &\equiv \frac{1}{2} \Big(x_p+\frac{y_p}{2}\Big)^j\sum^{j/2}_{n=0}\binom{j}{2n} \equiv 2^{j-2}\Big(x_p+\frac{y_p}{2}\Big)^j \not\equiv 0  \pmod{p} 
\end{align*}
\end{proof}
\begin{proposition}\label{2}
For prime 2, $\Lambda^2_D$ is nonempty only for $D=1,2,7$. In this case $a(D,j,2)$ does not vanish for all $j\in 2\Z^+$. Moreover, we have that $a(7,j,2)\equiv1\pmod{2}$
\end{proposition}
\begin{proof}
For $D=1,2$, $2|\Delta_{\mathcal{O}_D}(=-4D)$ so the ideal $(2)$ is ramified in $\mathcal{O}_D$, in particular there are elements of norm 2.  For $D\in\{3,7,11,19,43,67,163\}$, $2\nmid\Delta_{\mathcal{O}_D}(=-D)$. So the ideal $(2)$ is unramified in $\mathcal{O}_D$. Here we need to check whether 2 splits or not. We have the condition that 2 splits if and only if $-D\equiv 1\pmod{8}$. Only $D=7$ satisfies the condition. 

A brute force calculation shows that $a(1,j,2)=(1+i)^j\not=0$, $a(2,j,2)=i^j2^{j+1}\not=0,$ and $a(7,j,2)= 4{\rm Re}\Big(\frac{1+\sqrt{-7}}{2}\Big)^j\not=0.$ 

We prove that $a(7,j,2)\equiv1\pmod{2}$ using induction on even $j$. First, note that $a(7,2,2)=-3\equiv1\pmod{2}$. Now we assume that $a(7,j,2)\equiv1\pmod{2}$, which implies that ${\rm Re}\Big(\frac{1+\sqrt{-7}}{2}\Big)^j=(2k+1)/2$ for some $k$. The norm of $\Big(\frac{1+\sqrt{-7}}{2}\Big)^j$ is even, so we get that ${\rm Im}\Big(\frac{1+\sqrt{-7}}{2}\Big)^j=\sqrt{7}(2k'+1)/2$ for some $k'$. An easy calculation shows that $a(7,j+2,2)=-3{\rm Re}\Big(\frac{1+\sqrt{-7}}{2}\Big)^j-\sqrt{7}{\rm Im}\Big(\frac{1+\sqrt{-7}}{2}\Big)^j\equiv1\pmod{2}.$
\end{proof}
\subsection{Proof of Theorem \ref{Main}}
Proposition \ref{odd}, \ref{I} and \ref{R} together imply that $a(\Lambda_D,R_{D,j},r)$ and $a(\Lambda_D,I_{D,j},r)$ vanish for all $j\not\equiv 0\pmod{u_D}$, which implies that every nonempty shell $\Lambda^r_D$ is an ellipsoidal $T_D$-design (remember that $T_D=\Z^+\setminus u_D\Z^+$).

Now we prove the maximality of $T_D$. We show that $a(D,j,r)\not=0$ (note that $a(D,j,r)= \frac{1}{u_D}a(\Lambda_D,R_{D,j},r)$) for all $j\not\in T_D$ and $\Lambda^r_D$ nonempty. 
By Lemma \ref{Hecke prop}, enough to take $r$ to be a prime power. Suppose $p$ be a prime and $\alpha\geq 1$ be such that $\Lambda^{p^{a}}_D\not=\phi$. There are two cases possible, either $\Lambda^p_D$ is empty or it is not. First suppose $\Lambda^p_D$ is nonempty. 
If $p$ is 2 then $a(D,j,2)\not=0$ by Proposition \ref{2}. By part{\it (2)} of Lemma \ref{Hecke prop}, we have that $a(D,j,2^{\alpha})=a(D,j,2)^{\alpha}\not=0$ for $D=1,2$ since $\chi(2)=0$. When $D=7$ then part{\it (3)} of Lemma \ref{Hecke prop}, we have $a(7,j,2^{\alpha})\not=0$.
If $p$ is an odd prime, then Lemma \ref{non-zero} implies that $a(D,j,p)\not=0$. Now using part{\it (3)} of Lemma \ref{Hecke prop} again, we have $a(D,j,p^{\alpha})\not=0$. 
Suppose $\Lambda^p_D$ is empty then $a(D,j,p)=0$ and Lemma \ref{prime-decom} implies $\alpha$ is even. Now by part{\it (2)} of Lemma \ref{non-zero}, we get $a(D,j,p^{\alpha})=p^{j\alpha/2}\not=0$ (note that this case includes 2 too). So we get that $a(D,j,p^{\alpha})\not=0$ whenever $\Lambda^{p^{\alpha}}_D$ is nonempty.

\end{document}